\documentclass[11pt,a4paper]{article}
\usepackage[utf8]{inputenc}
\usepackage[left=1in, right=1in, top=1in, bottom=1in]{geometry}
\usepackage{amsmath}
\usepackage{amsfonts}
\usepackage{amssymb}
\usepackage{amsthm}
\usepackage{mathrsfs}
\usepackage{graphicx}
\usepackage{hhline}
\usepackage{color,xcolor}
\usepackage{caption}
\usepackage{enumerate}
\usepackage{mathtools}
\usepackage{hyperref}
\hypersetup{colorlinks = true, linkbordercolor = {white}}
\usepackage{textcomp}
\usepackage{aliascnt}
\usepackage{setspace}

\usepackage{tikz}
\usetikzlibrary{calc,decorations.pathreplacing,shapes,matrix,arrows,decorations.markings}
\tikzset{draw half paths/.style 2 args={%
  decoration={show path construction,
    lineto code={
      \draw [#1] (\tikzinputsegmentfirst) -- 
         ($(\tikzinputsegmentfirst)!0.5!(\tikzinputsegmentlast)$);
      \draw [#2] ($(\tikzinputsegmentfirst)!0.5!(\tikzinputsegmentlast)$)
        -- (\tikzinputsegmentlast);
    }
  }, decorate
}}
\tikzset{draw third paths/.style 2 args={%
  decoration={show path construction,
    lineto code={
      \draw [#1] (\tikzinputsegmentfirst) -- 
         ($(\tikzinputsegmentfirst)!0.33!(\tikzinputsegmentlast)$);
      \draw [#2] ($(\tikzinputsegmentfirst)!0.67!(\tikzinputsegmentlast)$)
        -- (\tikzinputsegmentlast);
    }
  }, decorate
}}
\tikzset{draw 2third paths/.style 2 args={%
  decoration={show path construction,
    lineto code={
      \draw [#1] (\tikzinputsegmentfirst) -- 
         ($(\tikzinputsegmentfirst)!0.67!(\tikzinputsegmentlast)$);
      \draw [#2] ($(\tikzinputsegmentfirst)!0.33!(\tikzinputsegmentlast)$)
        -- (\tikzinputsegmentlast);
    }
  }, decorate
}}

\tikzset{middlearrow/.style={
        decoration={markings,
            mark= at position 0.5 with {\arrow{#1}} ,
        },
        postaction={decorate}
    }
}

\tikzset{->-/.style={decoration={
  markings,
  mark=at position #1 with {\arrow{>}}},postaction={decorate}}}
  
\newtheorem{theorem}{Theorem}[section]
\newaliascnt{lemma}{theorem}
\newtheorem{lemma}[lemma]{Lemma} 
\aliascntresetthe{lemma}
\newaliascnt{proposition}{theorem}
\newtheorem{proposition}[proposition]{Proposition}   
\aliascntresetthe{proposition}
\newaliascnt{corollary}{theorem}
\newtheorem{corollary}[corollary]{Corollary}
\aliascntresetthe{corollary}
\newaliascnt{question}{theorem}
\newtheorem{question}[question]{Question}
\aliascntresetthe{question}
\newaliascnt{conjecture}{theorem}

\aliascntresetthe{conjecture}

\theoremstyle{definition}

\newtheorem*{notation}{Notation}
\newtheorem*{claim}{Claim}

\newaliascnt{example}{theorem}
\newtheorem{example}[example]{Example}
\aliascntresetthe{example}

\newaliascnt{definition}{theorem}

\aliascntresetthe{definition}

\newaliascnt{construction}{theorem}
\newtheorem{construction}[construction]{Construction}
\aliascntresetthe{construction}

\theoremstyle{remark}
\newtheorem*{remark}{Remark}

\newcommand{\mc}[1]{\mathcal{#1}}
\newcommand{\mb}[1]{\mathbb{#1}}

\newcommand{\vb}[1]{\mathbf{#1}}
\newcommand{\G}{\Gamma}
 
\renewcommand{\emptyset}{\varnothing}
\newcommand{\stc}{\; :\; }

\DeclareMathOperator{\Bi}{Bi}
\DeclareMathOperator{\Aut}{Aut}
\DeclareMathOperator{\Sym}{Sym}

\title{Every group-embeddable monoid arises as the bimorphism monoid of some graph}
\author{ Thomas D. H. Coleman\footnotemark[1]\; and Isaac K. Dilley\footnotemark[2]}

\begin{document}

\maketitle
\footnotetext[1]{Corresponding author. School of Mathematics and Statistics, University of St Andrews, St Andrews, KY16 9SS, United Kingdom. Email: \texttt{tdhc@st-andrews.ac.uk}.}
\footnotetext[2]{This article is based on work done during Isaac Dilley's summer Laidlaw scholarship at the University of St Andrews during the summers of 2020 and 2021. Isaac is currently working in industry.}

\begin{abstract}
Generalizing results of Frucht and de Groot/Sabidussi, we demonstrate that every group-embeddable monoid is isomorphic to the bimorphism monoid of some graph. \\[0.2em]

\end{abstract}

\section{Introduction and preliminaries}

Answering a question posed by K\"{o}nig in 1936 \cite{konig1936theorie}, Frucht proved that every finite group is isomorphic to the automorphism group of some finite graph; a result now known widely as \emph{Frucht's theorem} \cite{frucht1939herstellung}. Since then, there have been considerable efforts to generalize Frucht's theorem in a number of different directions. Frucht himself showed that every finite group arises as the automorphism group of a $3$-regular graph \cite{frucht1949graphs}. Sabidussi followed this by proving an analogous result for $k$-regular graphs and $k$-vertex-connected graphs (both for $k\geq 3$) and $k$-chromatic graphs (for $k\geq 2$) \cite{sabidussi1957graphs}, and Mendelsohn followed suit for strongly regular graphs \cite{mendelsohn1978every}. These theorems are part of a class of results known as \emph{universality theorems}, for which there are versions for other first-order structures (such as distributive lattices and commutative rings); see Babai for more \cite{babai1996automorphism}. Generalizing in a different direction (and working independently of each other), de Groot and Sabidussi proved that every group is isomorphic to the automorphism group of a graph \cite{de1959groups, sabidussi1960graphs}. A similar result is true for monoids; every monoid is isomorphic to the endomorphism monoid of some graph (see \cite{babai1996automorphism}).

Endomorphism monoids of a graph $\G$ encompass many different kinds of self-map, depending on the property of the underlying function on the vertex set $V\G$. In general, endomorphisms of $\Gamma$ must preserve edges by definition but may also add in edges or shrink a non-edge to a single point. By imposing additional properties of the underlying function of the endomorphism, behaviour like this is restricted. For instance, injective graph endomorphisms, called \emph{monomorphisms}, may only add in edges, but these may not shrink a non-edge to a single point as this would violate injectivity of the function. Of these endomorphisms, the most paradoxical (certainly for an algebraist!) is the idea of a \emph{bimorphism} of a graph $\Gamma$; a bijective endomorphism of $\Gamma$ that preserves edges but may change non-edges to edges -- and so is not necessarily an isomorphism. Since bimorphisms of $\G$ are permutations on the underlying vertex set $V\G$ (and hence elements of the symmetric group on $V\Gamma$), it follows that the monoid $\Bi(\G)$ of all bimorphisms of $\G$ is a \emph{group-embeddable monoid}.

Group-embeddable monoids have played an important part in the development of semigroup theory as a subject. For instance, a celebrated theorem of Ore states that a monoid $M$ is group-embeddable if and only if it is cancellative and satisfies Ore's condition on ideals (for $a,b\in M$ then $aM \cap bM \neq \emptyset$) (see \cite{clifford1961algebraic}). More recently, viewing group-embeddable monoids as submonoids of some symmetric group (called \emph{permutation monoids}) has led to a strong connection between group-embeddable monoids and bimorphisms of first-order structures. For instance, a submonoid of the infinite symmetric group $\Sym(\mb{N})$ is closed (under the pointwise convergence topology) if and only if it is isomorphic to the bimorphism monoid of some first-order structure \cite{coleman2019permutation}.

Given the celebrated results of Frucht and de Groot/Sabidussi, and the recent work on bimorphisms of first-order structures, it is natural to ask; \emph{is every group-embeddable monoid isomorphic to the bimorphism monoid of a graph}? In this article, we answer this question in the affirmative:

\begin{theorem}\label{maintheorem}
Let $M$ be a group-embeddable monoid. Then there exists a graph $\Gamma^*$ such that $\Bi(\Gamma^*)\cong M$.
\end{theorem}

Note that if $M$ is finite, it is a group; since $\Bi(\Gamma) = \Aut(\Gamma)$ for any finite graph $\Gamma$, this special case is precisely Frucht's theorem. We therefore restrict our attention to when $M$ is infinite. If $M$ is a group, then this reduces to the infinite version of Frucht's theorem proved by de Groot and Sabidussi. This theorem then generalizes both Frucht's original theorem and its generalization to infinite groups. 

We use standard graph-theoretic terminology throughout; a good source is \cite{diestel2000graph}. Unless specified, in this article all graphs are loopless, undirected, and unlabelled. A graph is \emph{rigid} if its automorphism group is trivial. Any monoid $M$ contains a \emph{group of units} which consist of all invertible elements of $M$; for more background information on semigroups and monoids, see \cite{howie1995fundamentals}. Functions act on the right of their arguments, and we compose maps from left to right.

\section{Strategy and motivating examples}

Let $G$ be a finite group. The proof of Frucht's theorem relies on the Cayley graph of $G$ with respect to some generating set $A$ of $G$ such that $e\notin A$; by definition, this is a loopless, labelled, directed graph. This Cayley graph is then turned into a simple undirected graph $\Gamma$ by the replacement of edges labelled by $a$ by a rigid `gadget' $R_a$. These gadgets are chosen to be pairwise non-isomorphic so if $a\neq b\in A$, then $R_a\ncong R_b$; this ensures that different generating elements are distinguished from each other. Since the gadgets are put in place to preserve the symmetries of the Cayley graph without adding any more, one can prove that $\Aut(\Gamma) \cong G$ via the action of $G$ by left multiplication on the modified Cayley graph; see \cite{cameron1994combinatorics} for finer details. This strategy generally extends to the infinite case; however, the gadgets tend to be far more intricate due to the nature of infinite cardinals. This is expanded upon in \autoref{brigid} for use throughout the article.

To test this strategy for group-embeddable monoids, one could look at the nicest group-embeddable monoid which is not a group; the \emph{infinite monogenic semigroup with identity} $(\mathbb{N}_0,+)$, which is embeddable via inclusion into the infinite cyclic group $(\mathbb{Z},+)$. It could be hoped that $\Bi(\Gamma)\cong (\mathbb{N}_0,+)$; but we will see now that this is not the case.

\begin{example}
Consider the Cayley graph $\Gamma'$ of $M = (\mathbb{N},+)$ with respect to the generating set $A = \{1\}$ of $M$, which is given in \autoref{1fig}.

\begin{center}
\begin{figure}[h]
\centering
\begin{tikzpicture}[node distance=2cm,inner sep=0.7mm,scale=1.7]
\node (-3b) at (-3,-0.5) [circle,draw,label=below:$0$] {};
\node (-2b) at (-2,-0.5) [circle,draw,label=below:$1$] {};
\node (-1b) at (-1,-0.5) [circle,draw,label=below:$2$] {};
\node (0b) at (0,-0.5) [circle,draw,label=below:$3$] {};
\node (1b) at (1,-0.5) [circle,draw,label=below:$4$] {};
\node (2b) at (2,-0.5) [circle,draw,label=below:$5$] {};
\node (3b) at (3,-0.5) [circle,draw,label=below:$6$] {};
\node (+b) at (4,-0.5) {};
\node (a) at (-2.5,-0.35) [color=black!30!white] {$1$};
\node (a) at (-1.5,-0.35) [color=black!30!white] {$1$};
\node (a) at (-0.5,-0.35) [color=black!30!white] {$1$};
\node (a) at (0.5,-0.35) [color=black!30!white] {$1$};
\node (a) at (1.5,-0.35) [color=black!30!white] {$1$};
\node (a) at (2.5,-0.35) [color=black!30!white] {$1$};

\draw[middlearrow={latex}] (-3b) -- (-2b);
\draw[middlearrow={latex}] (-2b) -- (-1b);
\draw[middlearrow={latex}] (-1b) -- (0b);
\draw[middlearrow={latex}] (0b) -- (1b);
\draw[middlearrow={latex}] (1b) -- (2b);
\draw[middlearrow={latex}] (2b) -- (3b);
\draw [thick,loosely dotted] (3b) -- (+b);
\end{tikzpicture}
\caption{Cayley graph of $M = (\mb{N},+)$ with respect to $A = \{1\}$.}\label{1fig}
\end{figure}
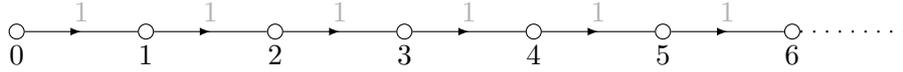
\end{center}
\vspace{-0.8cm}

Replacing each instance of the labelled directed edge corresponding to the single generator $1$ with an undirected edge gives a undirected graph $\Gamma$; see \autoref{2fig}. 

\begin{center}
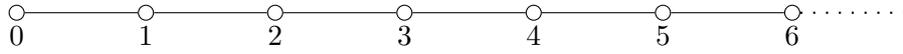
\begin{figure}[h]
\centering
\begin{tikzpicture}[node distance=2cm,inner sep=0.7mm,scale=1.7]

\node (-3b) at (-3,-0.5) [circle,draw,label=below:$0$] {};
\node (-2b) at (-2,-0.5) [circle,draw,label=below:$1$] {};
\node (-1b) at (-1,-0.5) [circle,draw,label=below:$2$] {};
\node (0b) at (0,-0.5) [circle,draw,label=below:$3$] {};
\node (1b) at (1,-0.5) [circle,draw,label=below:$4$] {};
\node (2b) at (2,-0.5) [circle,draw,label=below:$5$] {};
\node (3b) at (3,-0.5) [circle,draw,label=below:$6$] {};
\node (+b) at (4,-0.5) {};

\draw (-3b) -- (-2b);
\draw (-2b) -- (-1b);
\draw (-1b) -- (0b);
\draw (0b) -- (1b);
\draw (1b) -- (2b);
\draw (2b) -- (3b);
\draw [thick,loosely dotted] (3b) -- (+b);

\end{tikzpicture}
\caption{Graph $\Gamma$ obtained from the Cayley graph of $M = (\mb{N}_0,+)$ with respect to $A = \{1\}$.}\label{2fig}
\end{figure}
\end{center}
\vspace{-0.8cm}

Unfortunately, this graph has no non-trivial bijective endomorphisms. To see this, consider a bijection $\beta:\Gamma\to\Gamma$ that preserves edges; the aim is to prove by induction that $n\beta = n$ for all $n\in\mb{N}_0$. Since $\beta$ is a bijection, it must map some natural number to $0$; say $n$. If $n>0$, then $n$ has degree $2$, being adjacent to $n-1$ and $n+1$. As $0$ has degree $1$ and $\beta$ preserves edges, it follows that $n = 0$. Therefore, $\beta$ fixes $0$, completing the base case. For the inductive step, assume that $\beta$ fixes the first $k$ vertices. Since $\beta$ is a bijection, it must map some natural number $m$ to $k+1$; as $\beta$ preserves edges, it must send either $m-1$ or $m+1$ to $k$. Neither of these can happen as $k\beta = k$; so it follows that $(k+1)\beta = k+1$, completing the inductive step. This implies that $\Bi(\Gamma) \cong \{e\}$; in particular, this is not $(\mathbb{N},+)$ as hoped for.
\end{example}

So what was the problem here? To represent the group-embeddable monoid $M$ (that is not a group) as the bimorphism monoid of some graph $\Gamma$, we need to have bimorphisms of $\Gamma$ that are not automorphisms! So any element of $M$ not in its group of units should be represented by a bimorphism that changes a non-edge to an edge. To that end, there are a few considerations that need to be made when defining such a graph $\Gamma$:

\begin{itemize}
\item extra vertices need to be added to the Cayley graph of $M$ with respect to some generating set $A$ to allow things to move more freely;
\item bimorphisms of $\Gamma$ should add in edges where the group-embeddable monoid $M$ is not a group;
\item a way to identify those vertices representing elements of the group-embeddable monoid $M$ must be retained. 
\end{itemize}

This leads to our motivating example, first found in Coleman's PhD thesis \cite{coleman2017automorphisms}.

\begin{example}[\cite{coleman2017automorphisms}]\label{urexample}
Consider a graph $\G$ with vertex set $\mb{Z}\times \{0,1\}$, with adjacencies given by
\begin{itemize}
\item $(a,0)\sim (b,1)$ if and only if $a = b$;
\item $(a,0)\sim (b,0)$ if and only if $|a-b| = 1$, and;
\item $(a,1)\sim (b,1)$ if and only if $a\leq 0$ and $|a-b| = 1$.
\end{itemize}
See \autoref{3fig} for a diagram.

\begin{center}
\begin{figure}[h]
\centering
\begin{tikzpicture}[node distance=2cm,inner sep=0.7mm,scale=1.7]
\node (-a) at (-4,0.5) {};
\node (-3a) at (-3,0.5) [circle,draw,label=above:$(-3{,}1)$] {};
\node (-2a) at (-2,0.5) [circle,draw,label=above:$(-2{,}1)$] {};
\node (-1a) at (-1,0.5) [circle,draw,label=above:$(-1{,}1)$] {};
\node (0a) at (0,0.5) [circle,draw,label=above:$(0{,}1)$] {};
\node (1a) at (1,0.5) [circle,draw,label=above:$(1{,}1)$] {};
\node (2a) at (2,0.5) [circle,draw,label=above:$(2{,}1)$] {};
\node (3a) at (3,0.5) [circle,draw,label=above:$(3{,}1)$] {};
\node (+a) at (4,0.5) {};
\node (-b) at (-4,-0.5) {};
\node (-3b) at (-3,-0.5) [circle,draw,label=below:$(-3{,}0)$] {};
\node (-2b) at (-2,-0.5) [circle,draw,label=below:$(-2{,}0)$] {};
\node (-1b) at (-1,-0.5) [circle,draw,label=below:$(-1{,}0)$] {};
\node (0b) at (0,-0.5) [circle,draw,label=below:$(0{,}0)$] {};
\node (1b) at (1,-0.5) [circle,draw,label=below:$(1{,}0)$] {};
\node (2b) at (2,-0.5) [circle,draw,label=below:$(2{,}0)$] {};
\node (3b) at (3,-0.5) [circle,draw,label=below:$(3{,}0)$] {};
\node (+b) at (4,-0.5) {};
\draw (-3a) -- (-3b);
\draw (-2a) -- (-2b);
\draw (-1a) -- (-1b);
\draw (0a) -- (0b);
\draw (1a) -- (1b);
\draw (2a) -- (2b);
\draw (3a) -- (3b);
\draw [thick,loosely dotted] (-b) -- (-3b);
\draw (-3b) -- (-2b);
\draw (-2b) -- (-1b);
\draw (-1b) -- (0b);
\draw (0b) -- (1b);
\draw (1b) -- (2b);
\draw (2b) -- (3b);
\draw [thick,loosely dotted] (3b) -- (+b);
\draw [thick,loosely dotted] (-a) -- (-3a);
\draw (-3a) -- (-2a);
\draw (-2a) -- (-1a);
\draw (-1a) -- (0a);
\draw [thick,loosely dotted] (3a) -- (+a);
\end{tikzpicture}
\caption{Construction of $\G$ as in \autoref{urexample}.}\label{3fig}
\end{figure}
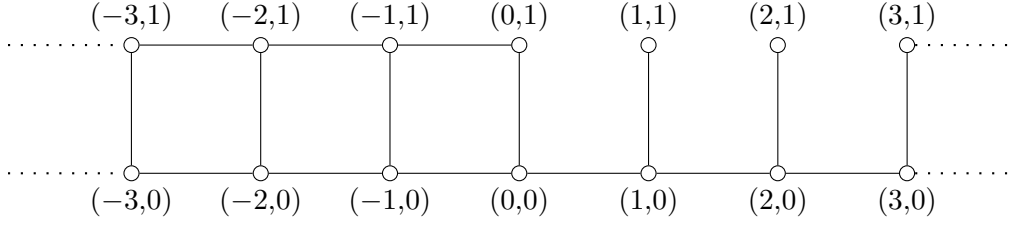
\end{center}
\vspace{-0.8cm}

As $(0,1)$ is the only vertex of degree $2$, any automorphism of $\G$ must fix $(0,1)$. In addition, an automorphism of $\G$ cannot swap $(0,0)$ and $(-1,1)$, as $(0,0)$ has a degree $1$ vertex at distance $2$ while $(-1,1)$ does not. So $(0,0)$ and $(0,1)$ are fixed, and it follows from this that every vertex is fixed and so $\G$ is rigid. 

However, there do exist bijections on $V\G$ such that only edges are preserved. Consider the shift map $\alpha:\G\to\G$ given by $(a,x)\alpha = (a-1,x)$. Then $\alpha$ preserves all edges and sends the non-edge between $(0,1)$ and $(1,1)$ to the edge between $(-1,1)$ and $(0,1)$. We claim that the only bimorphisms of $\G$ are of the form $\alpha^n$.

\begin{claim} Bi$(\G)\cong (\mb{N}_0,+)$, the infinite monogenic semigroup with identity.
\end{claim}

\begin{proof}[Proof of claim]
We show that the only bijective maps on vertices that preserve edges are of the form $\beta:\G\to\G$ such that $(b,x)\beta = (b-n,x)$; this is proved using a case analysis. First of all, there is no bimorphism $\alpha$ sending some vertex $(a,0)$ to $(0,1)$; this is because $\deg(a,0) > \deg(0,1)$ and every bimorphism of $\G$ preserves edges. Similarly, no bimorphism sends $(b,1)$ to $(0,1)$ for $b < 0$. So suppose that $\beta\in \Bi(\G)$ maps some $(b,1)$ to $(0,1)$ for $b \geq 0$. 

It follows that $\beta$ must send $(b,0)$ to either $(1,-1)$ or $(0,0)$ in order to preserve the adjacency $(b,0)\sim (b,1)$. Suppose that $(b,0)\beta = (-1,1)$. This gives rise to two cases:

\textbf{Case 1:} This is where $(b+1,0)\beta = (-2,1)$ and $(b-1,0)\beta = (-1,0)$. We consider the image point of $(b-2,0)$ under $\beta$. There are two choices; either $(b-2,0)\beta = (-2,0)$ or $(b-2,0)\beta = (0,0)$. Suppose initially that $(b-2,0)\beta = (-2,0)$. Since $\beta$ preserves edges, the only potential image point for the two vertices $(b-3,0)$ and $(b-2,1)$ (both adjacent to $(b-2,0)$) is $(-3,0)$; as $\beta$ is a bijection, this cannot happen. So now suppose that $(b-2,0)\beta = (0,0)$, the other potential choice. This means that $(1,0)$ is the only potential image point for both $(b-3,0)$ and $(b-2,1)$; another contradiction. Therefore, this case cannot occur.

\textbf{Case 2:} On the other hand, this is where $(b-1,0)\beta = (-2,1)$ and $(b+1,0)\beta = (-1,0)$. In this case, we consider the image point of $(b+2,0)$ under $\beta$; as above, this is either $(-2,0)$ or $(0,0)$. Assume that $(b+2,0)\beta = (-2,0)$. Here, this would leave $(-3,0)$ as the only potential image point for both vertices $(b+3,0)$ and $(b+2,1)$; this is a contradiction as $\beta$ is a bijection. So $(b+2,0)\beta = (0,0)$, which means that $(1,0)$ is the only potential image point for both $(b+3,0)$ and $(b+2,1)$. Therefore, this case also cannot occur; so it follows that $(b,0)\beta$ cannot be $(-1,1)$.

So $(b,0)\beta = (0,0)$ and hence $\beta$ maps the edge $(b,0)\sim (b,1)$ to the edge $(0,0)\sim (0,1)$. We can use another, similar case analysis to show that $(b+1,0)\beta = (1,0)$. Finally, we can show that $\beta$ cannot send $(a,0)$ to $(b,1)$ for $b < 0$ by another similar argument; this implies that $\beta$ must preserve the infinite two way line $\{(a,0)\; : \;a\in \mb{Z}\}$. All of these results together imply that $\beta$ must be a shift map and so the only bijective homomorphisms are of the form $(b,x)\beta = (b-n,x)$, for $n\geq 0$. It is not hard to see that $\alpha^n = \beta$ and thus Bi$(\G)$ has a single generator. Since any infinite semigroup that is generated by one element is isomorphic to $(\mb{N},+)$ (see \cite{howie1995fundamentals}), it follows that $\Bi(\G)\cong (\mb{N},+)$.

\end{proof}

You can notice that this graph $\Gamma$ contains an undirected version the Cayley graph for $(\mb{Z},+)$ with respect to the group generating set $A = \{1\}$ (the `bottom layer' of \autoref{3fig}) and some way of identifying an image of $(\mathbb{N}_0,+)$ in $(\mb{Z},+)$ in this graph (the one way line on the `top layer' of \autoref{3fig}). The undirected version of the Cayley graph of $(\mb{Z},+)$ provides the extra vertices to move the vertices corresponding to $(\mathbb{N}_0,+)$ around. Any non-trivial bimorphism of $\Gamma$ adds in edges on the `top layer'. Finally, $\Aut(\G)$ is trivial, which corresponds to the fact that the group of units of $(\mb{N},+)$ is trivial.
\end{example}

Following \autoref{urexample}, the rough steps to prove \autoref{maintheorem} are as follows:

\begin{enumerate}[(A)]
\item define a graph $\Gamma$ based on the Cayley graph for a group $G$ that $M$ embeds in (with respect to some generating set $A$ of $G$) as a `bottom layer';
\item show that this graph $\Gamma$ has no non-trivial bimorphisms, so $\Bi(\Gamma) = \Aut(\Gamma) = G$;
\item define a second graph $\Gamma^*$ that extends $\Gamma$ by adding a `top layer' of vertices corresponding to group elements of $G$ above $\Gamma$, while identifying an image of $M$ in this `second layer';
\item show that this graph $\Gamma^*$ is such that $\Bi(\Gamma^*)\cong M$, completing the proof.
\end{enumerate}

\section{Proof of \autoref{maintheorem}}

\subsection*{Bottom layer $\G$ [(A) and (B)]}

The first construction we give illustrate the gadgets required to define a graph with given automorphism group. Here, we use de Groot's inductive demonstration that for every cardinal $\vb{m}$, there exist $\vb{m}$ many non-isomorphic, rigid graphs \cite{de1959groups}.

\begin{construction}[\cite{de1959groups}]\label{brigid}
Step $0$: Suppose that $\{p,q\}$ is an isomorphic copy of $K_2$.

Step $1$: Let $\vb{m}>0$ be any cardinal. Add in $\vb{m}$ new vertices $\{q_{\alpha_1}\stc \alpha_1\in \vb{m}\}$, and draw each edge $\{q,q_{\alpha_1}\}$; here, the $1$ represents the step of the construction.

Step $n$: Assume that $n-1$ steps of this construction have been completed; that is, for all $1 < i < n$:
\begin{itemize}
\item there are cardinals $\vb{m}_{\alpha_1\alpha_2\ldots\alpha_i}$ such that if $\alpha_k\neq \beta_k$ for some $1\leq k \leq i$ then \[\vb{m}_{\alpha_1\alpha_2\ldots\alpha_i} \neq \vb{m}_{\beta_1\beta_2\ldots\beta_i},\] and where $\alpha_i\in\vb{m}_{\alpha_1\alpha_2\ldots\alpha_{i-1}}$. 
\item the vertices $q_{\alpha_1\alpha_2\ldots\alpha_i}$ and edges $\{q_{\alpha_1\alpha_2\ldots\alpha{i-1}},q_{\alpha_1\alpha_2\ldots\alpha_i}\}$ exist. Note that the vertices $q_{\alpha_1\alpha_2\ldots\alpha_{n-1}}$ exist and have degree $1$.
\end{itemize}

Now, for all $\alpha_1\alpha_2\ldots\alpha_{n-1}$ where $(\alpha_1,\alpha_2,\ldots,\alpha_{n-1})\in \vb{m}\times\vb{m}_{\alpha_1}\times \ldots \times \vb{m}_{\alpha_1\alpha_2\ldots\alpha_{n-2}}$:
\begin{itemize}
\item select a cardinal $\vb{m}_{\alpha_1\alpha_2\ldots\alpha_{n-1}}$ such that such that if $\alpha_k\neq \beta_k$ for some $1\leq k \leq i$ then \[\vb{m}_{\alpha_1\alpha_2\ldots\alpha_i} \neq \vb{m}_{\beta_1\beta_2\ldots\beta_i}\] noting also that $\vb{m}_{\alpha_1\alpha_2\ldots\alpha_{n-1}}\neq \vb{m}$ for all such sequences $\alpha_1\alpha_2\ldots\alpha_{n-1}$;
\item add in $\vb{m}_{\alpha_1\alpha_2\ldots\alpha_{n-1}}$ vertices $\{q_{\alpha_1\alpha_2\ldots\alpha_n}\stc \alpha_n\in\vb{m}_{\alpha_1\alpha_2\ldots\alpha_{n-1}}\}$ to the graph;
\item draw in all edges $\{q_{\alpha_1\alpha_2\ldots\alpha_{n-1}},q_{\alpha_1\alpha_2\ldots\alpha_n}\}$.
\end{itemize}
\end{construction}

An illustration is given in \autoref{4fig}. This process completes after countably many steps, and the resulting graph $\mc{G}$ has trivial automorphism group. By varying the choice of cardinals at each step, we can conclude that for every cardinal $\vb{m}$ there exists $\vb{m}$ many non-isomorphic graphs with trivial automorphism group. 

\begin{center}
\begin{figure}[h]
\centering
\begin{tikzpicture}[node distance=2cm,inner sep=0.7mm,scale=1.5]
\node (-3b) at (-3.5,-0.75) [circle,draw,label=below:$p$] {};
\node (-2b) at (-2.5,-0.5) [circle,draw,label=below:$q$] {};

\node (-1b) at (-3.5,0.5) [circle,draw] {};
\node (0b) at (-2.5,0.5) [circle,draw] {};
\node (1b) at (-1.5,0.5) [circle,draw] {};
\node (3b) at (-0.5,0.5) {};
\node (4b) at (-3.9,0.4) {$\mathbf{m}$};
\draw [thick,loosely dotted] (1b) -- (3b);
\draw [dotted] (-3.75,0.4) -- (-0.35,0.4) -- (-0.35,0.6) -- (-3.75,0.6) -- (-3.75,0.4);

\draw (-3b) -- (-2b);
\draw (-2b) -- (-1b);
\draw (-2b) -- (0b);
\draw (-2b) -- (1b);

\node (-1c) at (-5,1.5) [circle,draw] {};
\node (0c) at (-4.625,1.5) [circle,draw] {};
\node (1c) at (-4.5,1.5) {};
\node (2c) at (-4,1.5) [] {};
\node (3c) at (-5.4,1.4) {$\mathbf{m}_{\alpha_1}$};
\draw [thick,loosely dotted] (1c) -- (2c);
\draw [dotted] (-5.15,1.4) -- (-3.95,1.4) -- (-3.95,1.6) -- (-5.15,1.6) -- (-5.15,1.4);

\node (-1d) at (-3,1.5) [circle,draw] {};
\node (0d) at (-2.625,1.5) [circle,draw] {};
\node (1d) at (-2.5,1.5) {};
\node (2d) at (-2,1.5) [] {};
\node (3d) at (-3.4,1.4) {$\mathbf{m}_{\beta_1}$};
\draw [thick,loosely dotted] (1d) -- (2d);
\draw [dotted] (-3.15,1.4) -- (-1.95,1.4) -- (-1.95,1.6) -- (-3.15,1.6) -- (-3.15,1.4);

\node (-1e) at (-1,1.5) [circle,draw] {};
\node (0e) at (-0.625,1.5) [circle,draw] {};
\node (1e) at (-0.5,1.5) {};
\node (2e) at (0,1.5) [] {};
\node (3e) at (-1.4,1.4) {$\mathbf{m}_{\gamma_1}$};
\draw [thick,loosely dotted] (1e) -- (2e);
\draw [dotted] (-1.15,1.4) -- (0.05,1.4) -- (0.05,1.6) -- (-1.15,1.6) -- (-1.15,1.4);

\draw [thick, loosely dotted] (0.5,1.5) -- (1.25,1.5);
\draw (-1b) -- (-1c);
\draw (-1b) -- (0c);
\draw (0b) -- (-1d);
\draw (0b) -- (0d);
\draw (1b) -- (-1e);
\draw (1b) -- (0e);

\node (-1f) at (-6.5,2.5) [circle,draw] {};
\node (0f) at (-6.25,2.5) [circle,draw] {};
\node (2f) at (-5.75,2.5) [] {};
\node (3f) at (-6.9,2.7) {$\mathbf{m}_{\alpha_1\alpha_2}$};
\draw [thick,loosely dotted] (0f) -- (2f);
\draw [dotted] (-6.65,2.4) -- (-5.75,2.4) -- (-5.75,2.6) -- (-6.65,2.6) -- (-6.65,2.4);

\node (-1g) at (-5.35,2.5) [circle,draw] {};
\node (0g) at (-5.1,2.5) [circle,draw] {};
\node (2g) at (-4.6,2.5) [] {};
\node (3g) at (-5.7,2.7) {$\mathbf{m}_{\alpha_1\beta_2}$};
\draw [thick,loosely dotted] (0g) -- (2g);
\draw [dotted] (-5.45,2.4) -- (-4.55,2.4) -- (-4.55,2.6) -- (-5.45,2.6) -- (-5.45,2.4);

\draw (-1c) -- (-1f);
\draw (-1c) -- (0f);
\draw (0c) -- (-1g);
\draw (0c) -- (0g);
\draw [thick, loosely dotted] (-4.4,2.5) -- (-3.9,2.5);

\node (-1h) at (-3.5,2.5) [circle,draw] {};
\node (0h) at (-3.25,2.5) [circle,draw] {};
\node (2h) at (-2.75,2.5) [] {};
\node (3h) at (-3.9,2.7) {$\mathbf{m}_{\beta_1\alpha_2}$};
\draw [thick,loosely dotted] (0h) -- (2h);
\draw [dotted] (-3.65,2.4) -- (-2.75,2.4) -- (-2.75,2.6) -- (-3.65,2.6) -- (-3.65,2.4);

\node (-1i) at (-2.35,2.5) [circle,draw] {};
\node (0i) at (-2.1,2.5) [circle,draw] {};
\node (2i) at (-1.6,2.5) [] {};
\node (3i) at (-2.7,2.7) {$\mathbf{m}_{\beta_1\beta_2}$};
\draw [thick,loosely dotted] (0i) -- (2i);
\draw [dotted] (-2.45,2.4) -- (-1.55,2.4) -- (-1.55,2.6) -- (-2.45,2.6) -- (-2.45,2.4);

\draw (-1d) -- (-1h);
\draw (-1d) -- (0h);
\draw (0d) -- (-1i);
\draw (0d) -- (0i);
\draw [thick, loosely dotted] (-1.4,2.5) -- (-0.9,2.5);

\node (-1j) at (-0.5,2.5) [circle,draw] {};
\node (0j) at (-0.25,2.5) [circle,draw] {};
\node (2j) at (0.25,2.5) [] {};
\node (3j) at (-0.9,2.7) {$\mathbf{m}_{\gamma_1\alpha_2}$};
\draw [thick,loosely dotted] (0j) -- (2j);
\draw [dotted] (-0.65,2.4) -- (0.25,2.4) -- (0.25,2.6) -- (-0.65,2.6) -- (-0.65,2.4);

\node (-1k) at (0.65,2.5) [circle,draw] {};
\node (0k) at (0.9,2.5) [circle,draw] {};
\node (2k) at (1.4,2.5) [] {};
\node (3k) at (0.3,2.7) {$\mathbf{m}_{\gamma_1\beta_2}$};
\draw [thick,loosely dotted] (0k) -- (2k);
\draw [dotted] (0.55,2.4) -- (1.45,2.4) -- (1.45,2.6) -- (0.55,2.6) -- (0.55,2.4);

\draw (-1e) -- (-1j);
\draw (-1e) -- (0j);
\draw (0e) -- (-1k);
\draw (0e) -- (0k);
\draw [thick, loosely dotted] (1.6,2.5) -- (2.1,2.5);

\draw [thick, loosely dotted] (2.5,2.5) -- (3.25,2.5);

\draw [thick, loosely dotted] (-6.25,2.75) -- (-6.25,3.25);
\draw [thick, loosely dotted] (-5.1,2.75) -- (-5.1,3.25);
\draw [thick, loosely dotted] (-3.25,2.75) -- (-3.25,3.25);
\draw [thick, loosely dotted] (-2.1,2.75) -- (-2.1,3.25);
\draw [thick, loosely dotted] (-0.25,2.75) -- (-0.25,3.25);
\draw [thick, loosely dotted] (0.9,2.75) -- (0.9,3.25);
\end{tikzpicture}
\caption{Construction of the tree $\mc{G}$ in \autoref{brigid}.}\label{4fig}
\end{figure}
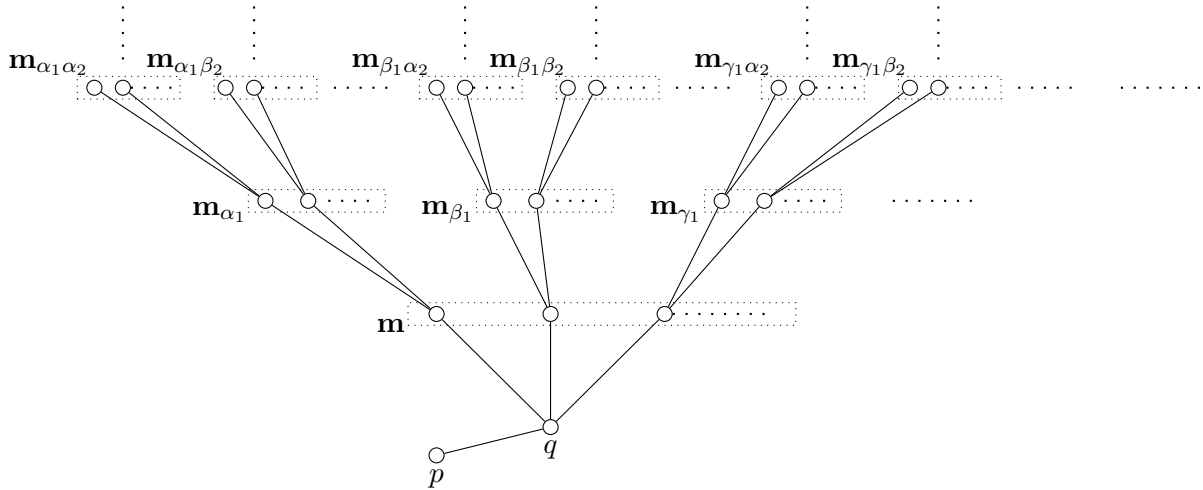
\end{center}
\vspace{-0.8cm}

Note that the number of steps in this construction cannot be finite; as if so, then the edges $\{q_{\alpha_1\alpha_2\ldots\alpha_{n-1}},q_{\alpha_1\alpha_2\ldots\alpha_n}\}$ with root $q_{\alpha_1\alpha_2\ldots\alpha_{n-1}}$ can be permuted by the symmetric group on $\vb{m}_{\alpha_1\alpha_2\ldots\alpha_{n-1}}$ elements. It follows from this that $p$ is the only vertex with degree $1$. Note also that $\mc{G}$ is a tree; there are no cycles induced at any point in the construction. 

We need to show that these rigid gadgets are indeed suitable for our uses by showing that they are \emph{bimorphism-rigid}; that is, the only bimorphism of the graph is the identity mapping. Recall (from \cite{coleman2019permutation}, for example) that two graphs $\mc{G}, \mc{H}$ are \emph{bimorphism equivalent} if there exist bijective homomorphisms $\alpha:\mc{G}\to\mc{H}$ and $\beta:\mc{H}\to\mc{G}$.

\begin{lemma}\label{nonbieq}
\begin{enumerate}[(a)]
\item Any two graphs $\mc{G}, \mc{H}$ obtained from \autoref{brigid} are bimorphism equivalent if and only if they are the same graph. As a corollary, every graph $\mc{G}$ obtained from \autoref{brigid} is bimorphism-rigid.

\item If $|\mc{G}|\neq |\mc{H}|$, then there is no pair of mutual monomorphisms $\alpha:\mc{G}\to \mc{H}$ and $\beta:\mc{H}\to\mc{G}$.
\end{enumerate}

\end{lemma}

\begin{proof}
(a) This proof relies on the fact that $\mc{G}$ and $\mc{H}$ are trees. Since $\mc{H}$ is maximally acyclic, any bijective homomorphism from $\mc{G}$ to $\mc{H}$ that adds an edge must induce a cycle somewhere in $\mc{H}$; a contradiction as $\mc{H}$ is a tree. Therefore, every bijective homomorphism from $\mc{G}$ to $\mc{H}$ must preserve non-edges, and hence must be an isomorphism. Since $\mc{G}$ and $\mc{H}$ are isomorphic if and only if $\mc{G} = \mc{H}$, we are done. The corollary follows immediately upon setting $\mc{G} = \mc{H}$ in the above proof.

(b) Follows from the Cantor-Schroeder-Bernstein theorem.
\end{proof}

\begin{remark}
As a corollary, note that $\Bi(\mc{G}) = \Aut(\mc{G})$ for any tree $\mc{G}$.
\end{remark}

Next, we use de Groot's construction to generate a Cayley graph $\Gamma$ of $G$ such that $\Aut(\Gamma) = \Bi(\Gamma) \cong G$, by replacing each of the usually labelled edges using bimorphism-rigid graphs from \autoref{brigid}.

\begin{construction}[\cite{de1959groups}]\label{cayley}
Let $G$ be a group with $|G|>3$, and take $A = G\setminus\{e\}$; this is certainly a generating set for $G$.

For each $a\in A$, use \autoref{brigid} to construct a graph $R_a$ where each cardinal involved is strictly greater than $|G|$. Using the work above, we can also ensure that $R_a\cong R_b$ if and only if $a = b$, and that $|R_a| \neq |R_b|$ for all $a\neq b\in A$. Denote a typical vertex of $R_a$ by $v_a$. The vertices $p_a,q_a\in R_a$ respectively correspond to the specially named vertices $p,q$ in \autoref{brigid}.

Let us now define the vertex set of $\Gamma$ by:
\[V\Gamma = G\cup \bigcup_{a\in A}\left\{G\times \left(VR_a\cup \{a\}\right)\right\}\]
so vertices of $\Gamma$ are group elements $g$, ordered pairs $(g,a)$, or ordered pairs $(g,v_a)$.

The edges of $G$ are defined in two parts for ease of reading. The first part $E_1$ consists of the edges that make up the undirected Cayley graph of $G$ with respect to $A$, and contains machinery to attach the gadget $R_a$ to any instance of $a$ in the graph:
\[E_1(\Gamma) = \left\{\{g,(g,a)\}, \{(g,a), (g,p_a)\}, \{(g,p_a),ga\}\stc g\in G, a\in A,p_a\in R_a\right\}\]
The second set of edges code in the edges of $R_a$ for all $a\in A$:
\[E_2(\Gamma) = \bigcup_{a\in A}\left\{\{(g,v_a),(g,w_a)\}\stc \{v_a,w_a\}\in E(R_a)\right\}\]
Then, define $E\Gamma = E_1(\Gamma) \cup E_2(\Gamma)$. Call $\Gamma = (V\Gamma, E\Gamma)$ the \emph{modified Cayley graph of $G$ with respect to $A$}; since $A$ is fixed throughout the rest of the article, we will refer to this as the \emph{modified Cayley graph of $G$}.
\end{construction}

\begin{remark}
The assumption that $|G| >3$ is to make sure that a degree argument in the proof of \autoref{biautG} is sound. The case where $|G| \leq 2$ (or indeed where $G$ is finite) can be handled by Frucht's theorem, and the fact that any finite group-embeddable monoid is a group.
\end{remark}

\begin{notation}
Denote the induced subgraph of $\Gamma$ isomorphic to $R_a$ associated with $g$ by $(g,R_a)$, and write \[R(g) = \{(g,R_a)\stc a\in A\}.\]
\end{notation}

\begin{center}
\begin{figure}[h]
\centering
\begin{tikzpicture}[node distance=2cm,inner sep=0.7mm,scale=2.5]
\node (-a) at (-1.2,0.1) {};
\node (4a) at (-1.2,-0.1) {};
\node (-b) at (2.2,0.1) {};
\node (4b) at (2.2,-0.1) {};
\node (-5a) at (-1.4,0.2) {};
\node (5a) at (-1.4,-0.2) {};
\node (-5b) at (2.4,0.2) {};
\node (5b) at (2.4,-0.2) {};
\node (-3a) at (-1,0) [circle,draw] {};
\node (-2a) at (0,0.5) [circle,draw] {};
\node (-1a) at (1,0.5) [circle,draw,fill=red!20!white] {};
\node (0a) at (2,0) [circle,draw] {};
\node (1a) at (0,-0.5) [circle,draw,fill=blue!20!white] {};
\node (2a) at (1,-0.5) [circle,draw] {};

\node (g) at (-1,-0.15) {$g$};
\node (ga) at (2,-0.15) {$ga$};
\node (gaa) at (-0.05,0.65) {\begin{small}$(g,a)$\end{small}};
\node (gaia) at (1.05,-0.65) {\begin{small}$(ga,a^{-1})$\end{small}};
\node (gaan) at (0.15,-0.35) {\begin{small}$(ga,p_{a^{-1}})$\end{small}};
\node (gaian) at (0.9,0.35) {\begin{small}$(g,p_a)$\end{small}};
\node (gaan) at (-0.85,-0.65) {\begin{small}$(ga,q_{a^{-1}})$\end{small}};
\node (gaian) at (1.75,0.65) {\begin{small}$(g,q_a)$\end{small}};

\draw (-3a) -- (-a);
\draw (-3a) -- (4a);
\draw (0a) -- (-b);
\draw (0a) -- (4b);
\draw (-3a) -- (-2a);
\draw (-2a) -- (-1a);
\draw (-1a) -- (0a);
\draw (0a) -- (2a);
\draw (2a) -- (1a);
\draw (1a) -- (-3a);
\draw[thick, loosely dotted] (-5a) -- (-a);
\draw[thick, loosely dotted] (5a) -- (4a);
\draw[thick, loosely dotted] (-5b) -- (-b);
\draw[thick, loosely dotted] (5b) -- (4b);

\node (gq) at (1.5,0.75) [circle,draw,fill=red!20!white] {};
\node (-1) at (1,1) [circle,draw,fill=red!20!white] {};
\node (-1ex) at (0.6,1.2) {};
\node (0) at (1.5,1) [circle,draw,fill=red!20!white] {};
\node (0ex) at (1.5,1.3) {};
\node (1) at (2,1) [circle,draw,fill=red!20!white] {};
\node (1ex) at (2.4,1.2) {};
\node (3) at (2.5,1) {};
\draw [thick,loosely dotted] (gq) -- (3);
\draw [thick,loosely dotted] (-1) -- (-1ex);
\draw [thick,loosely dotted] (0) -- (0ex);
\draw [thick,loosely dotted] (1) -- (1ex);

\node (geq) at (-0.5,-0.75) [circle,draw,fill=blue!20!white] {};
\node (-1e) at (-1,-1) [circle,draw,fill=blue!20!white] {};
\node (-1eex) at (-1.4,-1.2) {};
\node (0e) at (-0.5,-1) [circle,draw,fill=blue!20!white] {};
\node (0eex) at (-0.5,-1.3) {};
\node (1e) at (0,-1) [circle,draw,fill=blue!20!white] {};
\node (1eex) at (0.4,-1.2) {};
\node (3e) at (-1.5,-1) {};
\draw [thick,loosely dotted] (geq) -- (3e);
\draw [thick,loosely dotted] (-1e) -- (-1eex);
\draw [thick,loosely dotted] (0e) -- (0eex);
\draw [thick,loosely dotted] (1e) -- (1eex);

\draw (-1a) -- (gq);
\draw (-1) -- (gq);
\draw (0) -- (gq);
\draw (1) -- (gq);

\draw (1a) -- (geq);
\draw (-1e) -- (geq);
\draw (0e) -- (geq);
\draw (1e) -- (geq);

\end{tikzpicture}
\caption{A snippet of the graph $\G$ constructed in \autoref{cayley}. Vertices in $(g,R_a)$ are highlighted in red; vertices in $(ga,R_{a^{-1}})$ are highlighted in blue.}\label{5fig}
\end{figure}
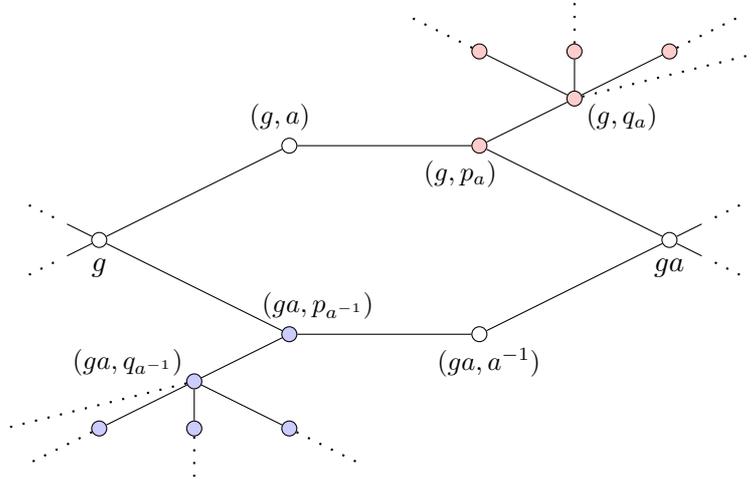
\end{center}
\vspace{-0.8cm}

\begin{lemma}\label{holefiller}
Let $\beta$ be a bimorphism of $\Gamma$. Then $(g,R_a)\beta = (g\beta,R_a)$ for all $a\in A$ and $g\in G$. 
\end{lemma}

\begin{proof}
Since each $R_a$ has a different cardinality by construction, both $R(g)$ and $R(g\beta)$ can be well ordered by cardinality of the $R_a$'s. Our claim is that $\beta$ induces a bijection from $R(g)$ to $R(g\beta)$ that preserves this well-order.

Since $g$ is mapped to $g\beta$, one can use a degree argument to show that each $(g,R_b)$ is mapped to some $(g\beta,R_c)$ for some $c\in A$. This must be at least a monomorphism; hence $\beta|_{R(g)}:R(g) \to R(g\beta)$ is a monomorphism.

As $\beta$ is a bimorphism, for every $(g\beta,p_b)$ there exists an $x\in \Gamma$ such that $x\beta = (g\beta,p_b)$. Because of this, the degree of $x$ must be less than or equal to $(g\beta,p_b)$, which is $3$. This means that the only potential preimage for $(g\beta,p_b)$ is either $(k,p_c)$ or $(k,c)$ for some $c$. 

Assume for a contradiction that it is $(k,c)$ for some $k\in G$ and $c\in A$; so that $(k,c)\beta = (g\beta,p_b)$. As $(k,c)$ has degree $2$, it follows that either $k$ or $(k,p_c)$ is mapped to $(g\beta,b)$; this is impossible by a degree argument. So any preimage of $(g\beta,p_b)$ must be $(k,p_c)$ for some $k\in G$ and $c\in A$. Since this happens for all $(g\beta,p_b)$, and each $R_a$ has cardinality greater than that of $G$, it follows that $\beta$ induces a bijection from $R(g)$ to $R(g\beta)$.

Now, to show that it preserves the stated well-order on $R(g)$. Assume for a contradiction that this is not true. Therefore, there is a least element in $R(g)$ (ordered by cardinality) such that $(g,R_a)\beta = (g\beta,R_b)$, where $R_b \neq R_a$. However, since $\beta$ is a bijection there must exist $(g,R_c)$ such that $(g,R_c)\beta = (g\beta,R_a)$; this cannot happen as $|R_c| > |R_a|$ by assumption. Therefore, this bijection must be order-preserving, and $(g,R_a)\beta = (g\beta,R_a)$.
\end{proof}

\begin{corollary}\label{betainside}
Let $\beta\in \Bi(\Gamma)$. Then $(g,a)\beta = (g\beta,a)$ and $(ga)\beta = g\beta a$ for all $g\in G$ and $a\in A$. 
\end{corollary}

\begin{proof}
It follows from \autoref{holefiller} that $(g,p_a)\beta = (g\beta,p_a)$. As $g$ is sent to $g\beta$, and $(g\beta,a)$ is the only vertex adjacent to both $g\beta$ and $(g\beta,p_a)$, it follows that $(g,a)\beta = (g\beta,a)$. Since this happens, together with the fact that $(g,q_a)\beta = (g\beta,q_a)$, the only potential image point for $ga$ under $\beta$ must be $g\beta a$; giving $(ga)\beta = g\beta a$ for all $g\in G$ and $a\in A$.
\end{proof}

\begin{proposition}\label{biautG}
Let $G$ be a group with $|G| > 3$, and suppose that $\Gamma$ is the modified Cayley graph of $G$ as defined in \autoref{cayley}. Then $\Bi(\Gamma) = \Aut(\Gamma) \cong G$.
\end{proposition}

\begin{proof}
Let $g,h\in G$. Define a permutation $\alpha:\Gamma\to\Gamma$ as follows:
\begin{itemize}
\item For any $x\in G$, define $x\alpha = hg^{-1}x$;
\item $(x,a)\alpha = (hg^{-1}x,a)$ for all $x\in G$ and $a\in A$;
\item $(x,v_a)\alpha = (hg^{-1}x,v_a)$ for all $x\in G$, $v_a\in R_a$ and $a\in A$.
\end{itemize}

Note that $g\alpha = hg^{-1}g = h$, and $\alpha$ preserves edges; hence $\alpha\in\Bi(\Gamma)$ and acts as left multiplication by $hg^{-1}$ on $G$. The idea is to show that this is the unique bimorphism sending $g$ to $h$; so suppose that $\beta \in\Bi(\Gamma)$ also sends $g$ to $h$. Note that if $g = h$, then $\alpha$ is the (necessarily unique) identity map.

%

Let $x\in G$ and choose $b= g^{-1}x$, which gives $x = gb$. Then, using \autoref{betainside} \[x\beta = (gb)\beta = g\beta b = hb = hg^{-1}x = x\alpha\] so $x\beta = x\alpha$ for all $x\in G$. From this point, it follows that $\beta$ behaves like $\alpha$ on all other vertices of $\Gamma$, and so the two are equal. 

Since $g,h$ were chosen arbitrarily in $G$, there exists a unique bimorphism $\alpha'$ such that $h\alpha' = g$ and $\alpha'$ acts as left multiplication by $gh^{-1}$ on $G$. Therefore, $\alpha' = \alpha^{-1}$ and so each bimorphism defined is an automorphism.

The goal is now to show that there are there are no bimorphisms sending $g\in G$ to any of the other vertices in $\Gamma$. First, note that each $g\in G$ is part of a $6$-cycle given by \[\{g,(g,g^{-1})\}, \{(g,g^{-1}),(g,p_{g^{-1}})\}, \{(g,p_{g^{-1}}),e\}, \{e,(e,g)\}, \{(e,g),(e,p_{g})\}, \{(e,p_g),g\}\] Since this cyclic path must be preserved by a bimorphism, it follows that $g$ cannot be sent to a vertex in the acyclic induced subgraph $(h,R_a)$ for all $a\in A$ and $h\in G$, with the possible exception of $(h,p_a)$.

Note that $d(g) \geq |A| \geq 3$ for all $g\in G$, since $g\sim (g,a)$ and $(h,p_a)\sim g$ for some $h\in G$. Since $d((h,a)) = 2$ and $d((h,p_a)) = 3$ for all $h\in G$ and $a\in A$, it follows that there does not exist a bimorphism sending $g$ to either $(h,a)$ or $(h,p_a)$ for all $h\in G$. Therefore, the only non-trivial bimorphisms send vertices $g\in G$ to other group elements $h\in G$ via left multiplication, at which point the rest of the function is determined by uniqueness of the defined bimorphism $\alpha$. Therefore, $\Bi(\Gamma) = \Aut(\Gamma) \cong G$.
\end{proof}

\subsection*{Top layer graph $\G^*$ [(C) and (D)]}

Let $M$ be a group-embeddable monoid; the aim now is to use $\Gamma$ to build a graph $\G^*$ such that $\Bi(\Gamma^*) \cong M$. From the general strategy, the idea is to add a second layer of vertices corresponding to group elements of $G$ onto $\Gamma$ (as defined in \autoref{cayley}), and identify those vertices in some image of $M$ in $G$. 

\begin{construction}\label{bigraph}
Let $G$ be a group, and suppose that $\Gamma$ is constructed as in \autoref{cayley}. Suppose that $M\leq G$ is a monoid that embeds in the group $G$. Choose $\sigma$ to be such an embedding, and define $B = \textrm{im }\sigma\cong M$. 

Define the vertex set of $\Gamma^*$ to be \[V\Gamma^* = V\Gamma \cup \{(g,\bullet)\; :\; g\in G\}\] essentially adding a layer of vertices `above' elements of the group $G$.

Define the edge set of $\Gamma^*$ to be 
\[E\Gamma^* = E\Gamma \cup \left\{\{g,(g,\bullet)\}\stc g\in G\right\} \cup \left\{\{(a,\bullet),(b,\bullet)\}\stc a\neq b\in B\right\}\]
consisting of
\begin{itemize}
\item all edges from $\Gamma$;
\item edges between corresponding elements of $G\subseteq V\Gamma$ and $G\times \{\bullet\}\subseteq V\Gamma^*$;
\item edges between all pairs of distinct elements of $B\times\{\bullet\}\subseteq V\Gamma^*$.
\end{itemize}
\end{construction}

\begin{center}
\begin{figure}[h]
\centering
\begin{tikzpicture}[node distance=2cm,inner sep=0.7mm,scale=2.5]
\node (-a) at (-1.2,0.1) {};
\node (4a) at (-1.2,-0.1) {};
\node (-b) at (2.2,0.1) {};
\node (4b) at (2.2,-0.1) {};
\node (-5a) at (-1.4,0.2) {};
\node (5a) at (-1.4,-0.2) {};
\node (-5b) at (2.4,0.2) {};
\node (5b) at (2.4,-0.2) {};
\node (-3a) at (-1,0) [circle,draw] {};
\node (-2a) at (0,0.25) [circle,draw] {};
\node (-1a) at (1,0.25) [circle,draw] {};
\node (0a) at (2,0) [circle,draw] {};
\node (1a) at (0,-0.25) [circle,draw] {};
\node (2a) at (1,-0.25) [circle,draw] {};

\node (g) at (-1,-0.15) {$g$};
\node (ga) at (2,-0.15) {$ga$};
\node (gaa) at (-0.05,0.4) {\begin{small}$(g,a)$\end{small}};
\node (gaia) at (1.05,-0.4) {\begin{small}$(ga,a^{-1})$\end{small}};
\node (gaan) at (0.15,-0.1) {\begin{small}$(ga,p_{a^{-1}})$\end{small}};
\node (gaian) at (0.9,0.1) {\begin{small}$(g,p_a)$\end{small}};
\node (gaan) at (-0.85,-0.4) {\begin{small}$(ga,q_{a^{-1}})$\end{small}};
\node (gaian) at (1.75,0.4) {\begin{small}$(g,q_a)$\end{small}};
\node (gb) at (-1,1) [circle,draw] {};
\node (gab) at (2,1) [circle,draw] {};
\draw (-3a) -- (gb);
\draw (0a) -- (gab);
\node (g) at (-1,1.15) {$(g,\bullet)$};
\node (ga) at (2,1.15) {$(ga,\bullet)$};
\draw (gb) -- (gab);
\node (gb1) at (-1.2,1) {};
\node (gb2) at (-1.4,1) {};
\node (gab1) at (2.2,1) {};
\node (gab2) at (2.4,1) {};
\draw (gb) -- (gb1) {};
\draw [thick, loosely dotted] (gb1) -- (gb2);
\draw (gab) -- (gab1) {};
\draw [thick, loosely dotted] (gab1) -- (gab2);

\draw (-3a) -- (-a);
\draw (-3a) -- (4a);
\draw (0a) -- (-b);
\draw (0a) -- (4b);
\draw (-3a) -- (-2a);
\draw (-2a) -- (-1a);
\draw (-1a) -- (0a);
\draw (0a) -- (2a);
\draw (2a) -- (1a);
\draw (1a) -- (-3a);
\draw[thick, loosely dotted] (-5a) -- (-a);
\draw[thick, loosely dotted] (5a) -- (4a);
\draw[thick, loosely dotted] (-5b) -- (-b);
\draw[thick, loosely dotted] (5b) -- (4b);

\node (gq) at (1.5,0.5) [circle,draw] {};
\node (-1) at (1.25,0.625) {};
\node (0) at (1.5,0.75) {};
\node (1) at (1.75,0.625) {};
\draw [thick,loosely dotted] (gq) -- (-1);
\draw [thick,loosely dotted] (gq) -- (0);
\draw [thick,loosely dotted] (gq) -- (1);

\node (geq) at (-0.5,-0.5) [circle,draw] {};
\node (-1e) at (-0.75,-0.625) {};
\node (0e) at (-0.5,-0.75) {};
\node (1e) at (-0.25,-0.625) {};
\draw [thick,loosely dotted] (geq) -- (-1e);
\draw [thick,loosely dotted] (geq) -- (0e);
\draw [thick,loosely dotted] (geq) -- (1e);

\draw (-1a) -- (gq);

\draw (1a) -- (geq);

\end{tikzpicture}
\caption{A snippet of the graph $\G^*$ constructed in \autoref{bigraph}. Since $(g,\bullet)$ and $(ga,\bullet)$ are adjacent, both $g$ and $ga$ are in $B$.}\label{6fig}
\end{figure}
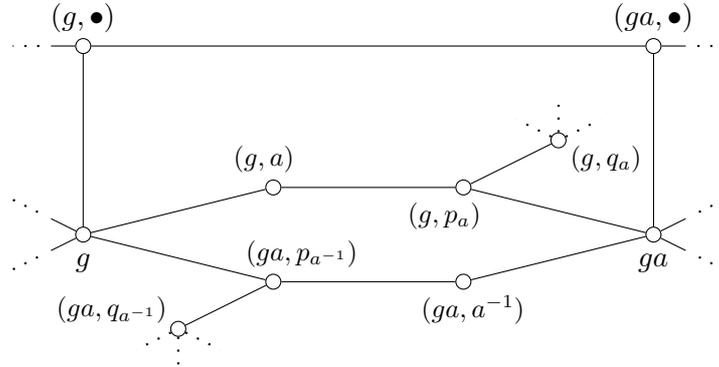
\end{center}
\vspace{-0.8cm}

\begin{lemma}\label{holefilleragain}
Let $\beta$ be a bimorphism of $\Gamma^*$. Then $(g,R_a)\beta = (g\beta,R_a)$ for all $g\in G$ and $a\in A$. 
\end{lemma}

\begin{proof}
Much like \autoref{holefiller}, our claim is that $\beta$ induces a bijection from $R(g)$ to $R(h)$ that preserves the well-order on cardinality in $R(g)$.

One can use a degree argument to show that each $(g,R_b)$ is mapped to some $(g\beta,R_c)$ for some $c\in A$. This must be at least a monomorphism; hence $\beta|_{R(g)}:R(g) \to R(g\beta)$ is a monomorphism.

As $\beta$ is a bimorphism, for every $(g\beta,p_b)$ there exists an $x\in \Gamma$ such that $x\beta = (g\beta,p_b)$. Because of this, the degree of $x$ must be less than or equal to the degree of $(g\beta,p_b)$, which is $3$. This means that the only potential preimages for $(g\beta,p_b)$ are some $(k,p_c)$ or $(k,c)$ for some $c$, or $(f,\bullet)$ for some $f\in G\setminus B$. Using a similar argument to \autoref{holefiller}, it follows that this potential preimage cannot be $(k,c)$.


So assume that $(f,\bullet)$ maps to $(g\beta,p_b)$. By a degree argument, it follows that $f$ must be mapped to $hb$. However, since $(f,\bullet)$ is mapped to $(g\beta,p_b)$, there is no other vertex adjacent to $f$ that can be mapped to the degree $1$ vertex $(hb,\bullet)$. Therefore, any preimage of $(g\beta,p_b)$ must be $(k,p_c)$ for some $k\in G$ and $c\in A$. Since this happens for all $(g\beta,p_b)$, and each $R_a$ has cardinality greater than that of $G$, it follows that $\beta$ induces a bijection from $R(g)$ to $R(g\beta)$.

Now, to show that $\beta$ preserves the well-order on $R(g\beta)$. Assume for a contradiction that this is not true. Therefore, there is a least element in $R(g\beta)$ (ordered by cardinality) such that $(g,R_a)\beta = (g\beta,R_b)$, where $R_b \neq R_a$. However, since $\beta$ is a bijection there must exist $(g,R_c)$ such that $(g,R_c)\beta = (g\beta,R_a)$; this cannot happen as $|R_c| > |R_a|$ by assumption. Therefore, this bijection must be order-preserving, and $(g,R_a)\beta = (g\beta,R_a)$.
\end{proof}

The proof of \autoref{betainside} applies in this case too, giving the following for free:

\begin{corollary}\label{betainsideagain}
Let $\beta\in \Bi(\Gamma^*)$. Then $(g,a)\beta = (g\beta,a)$ and $(ga)\beta = g\beta a$ for all $g\in G$ and $a\in A$. \qed
\end{corollary}


Finally, we can prove \autoref{maintheorem}. In the following, we will abuse notation and write $(C,\bullet)$ to mean the induced subgraph of $V\G^*$ on some $C\times \{\bullet\}\subseteq G\times\{\bullet\}$.


\begin{proof}[Proof of \autoref{maintheorem}]
Suppose that $B$ is the image of some group-embeddable monoid $M$ in a group $G$, and that $\G^*$ is the corresponding graph as constructed in \autoref{bigraph}. Let $g,h\in B$. Define a permutation $\alpha:\Gamma^*\to\Gamma^*$ as follows:
\begin{itemize}
\item For any $x\in G$, define $x\alpha = hg^{-1}x$;
\item $(x,a)\alpha = (hg^{-1}x,a)$ for all $x\in G$ and $a\in A$;
\item $(x,v_a)\alpha = (hg^{-1}x,v_a)$ for all $x\in G$, $v_a\in R_a$ and $a\in A$.
\item $(x,\bullet)\alpha = (hg^{-1}x,\bullet)$ for all $x\in G$.
\end{itemize}

If there exists $a\in B$ such that $ag = h$, then $ag = hg^{-1}g$. As group-embeddable monoids are cancellative, it follows that $a = hg^{-1}$. It follows that left multiplication in $B$ can be viewed as multiplication by those $hg^{-1}\in B$. Note that $\alpha$ acts as left multiplication by $hg^{-1}$ on $B$, should that element exist in $B$. Finally, note that if $g = h$, then $\alpha$ is the identity map. 

To prove the result, it is enough to show:

\begin{enumerate}[(1)]
\item $\alpha$ is a bimorphism of $\Gamma^*$ if and only if $hg^{-1}\in B$, and so the set of bimorphisms defined above act as left multiplication on $B$;
\item $\alpha$ is unique: so if $\beta$ is a bimorphism sending $g$ to $h$, with $g,h\in B$, then $\beta = \alpha$;
\item all bimorphisms of $\Gamma^*$ arise in this way. 
\end{enumerate}

For (1), \autoref{biautG} shows that $\alpha|_\Gamma$ is a bimorphism of $\Gamma$ regardless of the location of $hg^{-1}$ in $G$. Therefore, the statement of (1) reduces to proving that $\alpha$ preserves edges in $E\Gamma^*\setminus E\Gamma$ if and only if $hg^{-1}\in B$.

Assume that $hg^{-1}\in B$ and consider the edge $\{x,(x,\bullet)\}$. This is then sent by $\alpha$ to the pair $\{hg^{-1}x, (hg^{-1}x,\bullet)\}$; since $hg^{-1}x\in G$, this pair is an edge of $\Gamma^*$ by definition. Next, look at the edge $\{(a,\bullet), (b,\bullet)\}\in E\Gamma^*$; as this exists, it follows that $a,b\in B$. Since this happens, then $hg^{-1}a,hg^{-1}b\in B$ by assumption. It follows by definition that $\{(hg^{-1}a,\bullet),(hg^{-1}b,\bullet)\}$ is also an edge of $\Gamma^*$. This, together with \autoref{biautG}, shows that $\alpha$ is a bimorphism of $\Gamma^*$ when $hg^{-1}\in B$.

Next, assume that $hg^{-1}\notin B$. Since $|B|\geq 2$, it follows that there exists $y\in B$ that is not the identity element. Therefore, $(e,\bullet)\sim (y,\bullet)$ is an edge of $\Gamma^*$. However, $(e,\bullet)\alpha = (hg^{-1},\bullet)$ by definition; since $hg^{-1}\notin B$, it follows that $(hg^{-1},\bullet)\nsim (hg^{-1}y,\bullet)$ in $\Gamma^*$. So $\alpha$ is not a bimorphism of $\Gamma^*$ when $g^{-1}h\notin B$. This completes the proof of (1).

Now for (2). Assume that $\beta$ is another bimorphism of $\Gamma^*$ sending $g$ to $h$ in $B$. By \autoref{holefilleragain}, $\beta$ moves $(g,R_a)$ to the isomorphic induced subgraph $(g\beta,R_a) = (h,R_a)$ for all $a\in A$; it follows that $(ga)\beta = ha$ for all $a\in A$ by \autoref{betainsideagain}.

In addition, since every vertex in $\Gamma$ is adjacent to at least one vertex in some $(g,R_a)$, and every vertex in $V\Gamma^*\setminus V\Gamma$ is non-adjacent to every $(g,R_a)$, then no bimorphism sends a vertex of $V\Gamma$ to a vertex of $V\Gamma^*\setminus V\Gamma$. So the restriction $\beta|_{V\Gamma}$ of a bimorphism $\beta\in \Bi(\Gamma^*)$ to $\Gamma$ must be a monomorphism of $\Gamma$ at the very least. To prove that this must be a bimorphism of $\Gamma$, we must show that the preimage of any vertex of $\Gamma$ also lies in $\Gamma$. To do this, it is enough to prove that any vertex in $V\Gamma^*\setminus V\Gamma$ does not map into $V\Gamma$.

Since every vertex $(y,\bullet)$ of $V\Gamma^*\setminus V\Gamma$ is adjacent to $y\in\Gamma$, this edge must be preserved in any bimorphism $\beta$ sending $(y,\bullet)$ to some vertex of $V\Gamma$. This splits into two cases:
\begin{itemize}
\item If $y\in B$, then $(y,\bullet)$ is part of a complete graph of size at least $3$ contained in the induced subgraph $(B,\bullet)$, and this triangle must be preserved by any bimorphism $\beta$. However, the induced subgraph on $\Gamma^*\setminus (B,\bullet)$ contains no triangles (or indeed any complete graphs of size $3$ or larger). Therefore, elements of $(B,\bullet)$ cannot be mapped into $\Gamma$. As $(y,\bullet)\beta\in V\Gamma^*\setminus V\Gamma$ and the edge $\{y,(y,\bullet)\}$ must be preserved, the only element of $V\Gamma^*\setminus V\Gamma$ adjacent to $y\beta$ is $(y\beta,\bullet)$; and so $(y,\bullet)\beta = (y\beta,\bullet)$.

\item If $y\in G\setminus B$, then $(y,\bullet)$ has degree $1$ and is adjacent to $y\in V\Gamma$. Since $(y,\bullet)\beta\in V\Gamma$, and $y\beta\in V\Gamma$ by above, it follows that exactly one of the three cases must occur:
\begin{enumerate}[(a)]
 \item $(y,\bullet)\beta = (y\beta,\bullet)$;
 \item $(y,\bullet)\beta = (y\beta,a)$ for some $a\in A$ or
 \item $(y,\bullet)\beta = (y\beta b,p_{b^{-1}})$ for some $b\in A$.
\end{enumerate}
Of these, it cannot be case (b); since $(y,a)\beta = (y\beta,a)$ by \autoref{betainsideagain}. Also, it cannot be case (c); set $g = yb$ and $a = b^{-1}$ in \autoref{holefilleragain} and use \autoref{betainsideagain} to see that $(yb,p_{b^{-1}})\beta  = (yb\beta,p_{b^{-1}}) = (y\beta b, p_{b^{-1}})$. Therefore, it must be case (a) and so elements of $(G\setminus B, \bullet)$ cannot be mapped into $\Gamma$.
\end{itemize}

It follows that $\beta|_{V\Gamma}$ is a bimorphism of $V\Gamma$; so by \autoref{biautG}, $\beta = \alpha$ on $V\Gamma$. Work above also shows that \[(x,\bullet)\beta = (x\beta,\bullet) = (x\alpha,\bullet) = (hg^{-1}x,\bullet) = (x,\bullet)\alpha\] Therefore, $\beta = \alpha$ on $\Gamma^*$ as well.

Finally, for (3). It is enough to show that $g\in G$ is mapped to some other $x\in G$. By the work above, we know that $g$ cannot be moved to anywhere in $V\Gamma^*\setminus V\Gamma$ by any bimorphism; work in the proof of \autoref{biautG} then applies to show that $g$ can only be mapped to another group element in $G\subseteq V\Gamma$. This then completes the proof, as any bimorphism mapping $g$ to $x$ is unique and acts like left multiplication.
\end{proof}

Following Ore's theorem, the bimorphism monoid of a graph $\Gamma$ is a natural example of a \emph{cancellative monoids}. A monoid $M$ is \emph{left-cancellative} if for all $a,b,c\in M$, then $ab = ac$ implies $b=c$; equivalently, a monoid is left-cancellative if the action by $M$ on itself by left multiplication is injective. \emph{Right-cancellative monoids} are defined analogously, and a monoid is cancellative if and only if it is both left- and right-cancellative \cite{clifford1961algebraic}. Examples of left-cancellative monoids include \emph{epimorphism monoids} of graphs (the monoid of all surjective endomorphisms), and examples of right-cancellative monoids include \emph{monomorphism monoids} of graphs (the monoid of all injective endomorphisms) \cite{coleman2017automorphisms}. Therefore, there is a natural question to ask:

\begin{question}\label{q1}
Does every left-cancellative (right-cancellative) monoid arise as the epimorphism (monomorphism) monoid of a graph $\Gamma$?
\end{question}

We note that a similar question is \emph{not necessarily} true for bimorphism monoids, as not every cancellative monoid is group-embeddable. We therefore refine \autoref{q1}:

\begin{question}\label{q2}
What class of left-cancellative (right-cancellative) monoids can be represented as epimorphism (monomorphism) monoid of a graph $\Gamma$?
\end{question}

\end{document}